\documentclass[11pt,a4paper]{article}
\usepackage{epsf,epsfig,amsfonts,amsgen,amsmath,amstext,amsbsy,amsopn,amsthm,lineno}
\usepackage{color}
\usepackage{graphicx}
\usepackage{subfigure}
\usepackage{epstopdf}
\usepackage{authblk}
\usepackage{amsmath,amssymb,amsmath,amsfonts,amscd}
\usepackage{indentfirst,graphicx,epstopdf}
\usepackage{psfrag,ifpdf,enumerate}
\usepackage{caption}
\usepackage{color}
\usepackage[bookmarksnumbered, plainpages]{hyperref}
\usepackage[numbers,sort&compress]{natbib}
\usepackage[noend]{algpseudocode}
\usepackage{algorithmicx,algorithm}
\usepackage{psfrag}
\usepackage{color}
\usepackage{enumerate}
\usepackage{epstopdf}

\newtheorem{theorem}{Theorem}

\newtheorem{prop}{Proposition}
\newtheorem{lemma}{Lemma}

\theoremstyle{definition}
\newtheorem{definition}{Definition}
\newtheorem{claim}{Claim}

\newtheorem{conjecture}{Conjecture}

\newtheorem{question}{Question}
\newtheorem{case}{Case}

\newtheorem{lemmaA}{Lemma}
\newtheorem{claimA}[lemmaA]{Claim}

\title{\Large \bf The absence of monochromatic triangle implies various properly colored spanning trees}

\begin{document}
\date{}
\author[1,2]{\small Ruonan Li\thanks{Supported by National Natural Science Foundation of China (No. 11901459), China Scholarship Council (No. 202306290113) and the Institute for Basic Science (IBS-R029-C4). E-mail: rnli@nwpu.edu.cn}}
\author[1,2]{\small Ruhui Lu\thanks{E-mail:~lurh@mail.nwpu.edu.cn}}
\author[1,2]{\small Xueli Su\thanks{E-mail:~suxueli@nwpu.edu.cn}}
\author[1,2]{\small Shenggui Zhang\thanks{Supported by National Natural Science Foundation of China (No. 12071370, No. 12131013) and Shaanxi Fundamental Science Research Project for Mathematics and Physics (No.22JSZ009). E-mail:~sgzhang@nwpu.edu.cn}}
\affil[1]{ School of Mathematics and Statistics,
Northwestern Polytechnical University,
Xi'an, Shaanxi, 710129, P.R.~China}
\affil[2]{Xi'an-Budapest Joint Research Center for Combinatorics\\
Northwestern Polytechnical University, Xi'an, Shaanxi,710129, P.R.~China}
\maketitle

\begin{abstract}
An edge-colored graph $G$ is called properly colored if every two adjacent edges are assigned different colors. A monochromatic triangle is a cycle of length 3 with all the edges having the same color. Given a tree $T_0$, let $\mathcal{T}(n,T_0)$ be the collection of $n$-vertex trees that are subdivisions of $T_0$. It is conjectured that for each fixed tree $T_0$, there is a function $f(T_0)$ such that for each integer $n\geq f(T_0)$ and each $T\in \mathcal{T}(n,T_0)$, every edge-colored complete graph $K_n$ without containing monochromatic triangle must contain a properly colored copy of $T$. We confirm the conjecture in the case that $T_0$ is a star. A weaker version of the above conjecture is also obtained. Moreover, to get a nice quantitative estimation of $f(T_0)$ when $T_0$ is a star requires determining the constraint Ramsey number of a monochromatic triangle and a rainbow star, which is of independent interest.

\end{abstract}

\section{Introduction}
Let $G$ be an edge-colored graph. We say $G$ is {\it monochromatic} if all the edges are of the same color, and {\it rainbow} if all the edges are of distinct colors, and {\it lexical} if there is a total order of $V(G)$ such that edges $uv$ and $xy$ have a same color if and only if $\min\{u,v\}=\min\{x,y\}$. The {\it Canonical Ramsey Theorem} founded by Erd\H{o}s and Rado \cite{erdos_combinatorial_1950} states that given an integer $k$, each edge-colored complete graph of sufficiently large order always contains a $k$-clique $H$, which is either monochromatic or rainbow or lexical. In fact, the absence of certain rainbow (monochromatic) subgraph often forces a giant connected subgraph that is monochromatic (colors fully mixed). For instance, Erd\H{o}s and Rado \cite{erdos_combinatorial_1950} observed that every $2$-colored complete graph always contains a monochromatic spanning tree. Gallai partition theorem \cite{gallai_transitiv_1967}(c.f.\cite{gyarfas_edge_2004}) implies that the absence of rainbow triangle in edge-colored complete graphs forces a spanning 2-colored subgraph.
Studying a conjecture proposed by Andersen \cite{andersen_hamilton_1989}, Alon, Pokrovskiy and Sudakov \cite{alon_random_2017} proved that each edge-colored $K_n$ without monochromatic path of length $2$ contains a rainbow path of length $n-o(n)$. Asymptotically solving a conjecture given by Bollob\'{a}s and Erd\H{o}s\cite{bollobas_alternating_1976}, Lo's result \cite{lo_properly_2016} tells that each edge-colored $K_n$ without monochromatic star of size $(1/2-\epsilon)n$ always contains a properly colored Hamilton cycle when $n$ is sufficiently large. An edge-colored graph is called {\it properly colored} (say ``{\it PC }'' for short) if every two adjacent edges are of distinct colors.
In this paper, we continue the exploration of spanning properly colored subgraphs in edge-colored complete graphs when a fixed monochromatic configuration is forbidden.

An easy observation given by Barr \cite{barr_properly_1998} tells that each edge-colored $K_n$ without containing a monochromatic triangle (say ``mono-$C_3$-free'' for short) must contain a PC Hamilton path. Under the same condition, the first author \cite{li_properly_2021} studied the existence of PC Hamilton cycles and obtained a full characterization of the counterexamples. Particularly, when $n\geq 6$, every counterexample acts locally like a non-strongly-connected directed graph. In fact, for any PC spanning target graphs containing cycles, an edge-colored complete graph transformed from a transitive tournament is always a counterexample. Therefore it is natural to study the existence of properly colored spanning trees in mono-$C_3$-free complete graphs. For more relations between edge-colored graphs and directed graphs, we refer the readers to \cite{fujita_decomposing_2019, li_classification_2020, li_vertexdisjoint_2020, aharoni_rainbow_2019, ding_properly_2022}.

Let $T_0$ be a fixed tree. We use $\mathcal{T}(n,T_0)$ to denote the collection of $n$-vertex trees that are subdivisions of $T_0$. Note that a Hamilton path can be regarded as a subdivision of $K_2$. We show the existence of a properly colored spanning tree which is a subdivision of a given tree.

\begin{theorem}\label{Thm:subdivision of tree}
  Let $T_0$ be a tree of $k$ edges and let $G$ be a mono-$C_3$-free edge-colored $K_n$ with $n\geq (k+2)!$. Then there exists a tree $T\in \mathcal{T}(n,T_0)$ such that $G$ contains a PC copy of $T$.
\end{theorem}

In the above theorem, there is no control on the distribution of subdividing vertices on different edges of $T_0$. We wonder the existence of all possible subdivisions and therefore propose the following conjecture.
\begin{conjecture}\label{conj:spanning tree}
Let $T_0$ be a fixed tree. Then there is a function $f(T_0)$ such that every mono-$C_3$-free edge-colored $K_n$ with $n\geq f(T_0)$ contains a PC copy of $T$ for each tree $T\in \mathcal{T}(n,T_0)$.
\end{conjecture}
We confirm the above conjecture when $T_0$ is a star.
For $k\geq 3$, a subdivision of a $k$-star is called a {\it $k$-spider}. A {\it leg} in a $k$-spider $T$ is a path $P$ from the unique $k$-degree vertex to a leaf of $T$. If $P$ contains $\ell$ edges, then we say the leg $P$ is of length $\ell$. Use $C_3$ and $S_k$ to denote a triangle and a star of $k$ edges. Let $g(S_k,C_3)$ to be the maximum number $N$ such that there exists an edge-colored $K_N$ containing neither a rainbow $S_k$ nor a monochromatic $C_3$. The existence of  $g(S_k,C_3)$ is guaranteed by the Canonical Ramsey Theorem. In literature, Gy\'{a}rf\'{a}s, Lehel, Schelp and Tuza \cite{gyarfas_ramsey_1987} studied the {\it Local Ramsey number}; Jamison, Jiang and Ling \cite{jamison_constrained_2003} defined the {\it Constraint Ramsey Number}. Both definitions are consistent with $g(S_k,C_3)$. We will use the notation $g(S_k,C_3)$ to state our main result. A weak upper bound of  $g(S_k,C_3)$ is given in Section \ref{sec:4}, which is applied in the proof of Theorem \ref{Thm:subdivision of tree}.
\begin{theorem}\label{Thm:k-spider}
     Given an integer $k\geq 3$, let $G$ be a mono-$C_3$-free edge-colored $K_n$. If $n\geq 6k\cdot g(S_k,C_3)+2k^3+2k^2+8k$,  then
    for every $k$ positive integers $\ell_1\geq \ell_2\geq \cdots\geq \ell_k$ satisfying $\sum_{i=1}^k \ell_i=n-1$, $G$ contains a properly colored spanning spider $T$ with legs of lengths $\ell_1,\ell_2,\ldots,\ell_k$, respectively.
\end{theorem}
\noindent \textbf{Sketch of the proofs}. To prove Theorem \ref{Thm:subdivision of tree}, we first show the existence of every properly colored tree on $k$ edges (this is guaranteed by the Canonical Ramsey Theorem for sufficiently large $n$. Our proof gives an explicit bound for $n$), and then embed the remaining vertices greedily. However, the greedy method can not guarantee the number of subdividing vertices on each edge precisely as whatever we want. Therefore the key point in the proof of Theorem \ref{Thm:k-spider} is using some extensible structures as glue to merge vertices in mono-$C_3$-free complete graphs into legs. If the host graph has many ``nice bowties'', then we are home. Otherwise, by removing constant number of vertices, the host graph is essentially a multipartite tournament with certain properties inherent from the ``mono-$C_3$-free'' condition. Then by analyzing the structure of this multipartite tournament, we obtain a certain oriented tree which is almost the spanning tree we desired, except for the first leg. Applying some structural lemmas proved in  Section \ref{sec:2}, we finally embed the remaining vertices.

In Section \ref{sec:2}, we define some crucial structures and obtain related properties. In Section \ref{sec:3}, we study the ``mono-$C_3$-free tournament'' and show the existence of a certain oriented tree which is almost spanning. The proofs of Theorems \ref{Thm:subdivision of tree} and \ref{Thm:k-spider} are deliverd in Section \ref{sec:4}.
\section{Preliminaries}\label{sec:2}
\subsection{Related notions}

Let $G$ be an undirected graph. For each $U\subseteq V(G)$, we write $G-U$ for
$G[V(G) \setminus U]$.  If  $U=\{v\}$ is a singleton,
we write $G- v$ rather than $G-\{v\}$. Instead of $G-V (G^{\prime}  )$ we simply
write $G-G^{\prime}$.  For a set of edges $F$, we define $G-F:= (V(G), E(G)\setminus F)$ and $G+F:= (V(G)\cup V(F), E(G) \cup F)$;
as above, $G-\{e\}$ and $G+\{e\}$ are abbreviated to
$ G-e$ and $G+e$. Let $P$ be a path in $G$. The length of $P$ is the number of edges on $P$.

Let $G$ be an edge-colored graph. Denote by $col(e)$ and $col(G)$, respectively, the color of an edge $e$ and the set of colors assigned to $E(G)$.
For a vertex $v\in V(G)$, the \textit{color degree} of $v$ in $G$, denoted by $d^{c}_{G}(v)$ is the number of distinct colors assigned to the edges incident to $v$.
We use $\delta^{c}(G)=\min\{d^{c}_{G}(v): v\in V(G)\}$ to denote the \emph{minimum color degree} of $G$, and $\Delta^{mon}_G(v) = \max_{c\in col(G)}|\{u\in V(G)\setminus\{v\}: col(uv) = c\}|$ to denote the \emph{maximum monochromatic degree} of $v$.
For two disjoint subsets $V_1$ and $V_2$ of $V(G)$, denote by $col(V_1, V_2)$ the set of colors appearing on the edges between $V_1$ and $V_2$ in $G$. When $V_1=\{v\}$, use $col(v, V_2)$ to denote $col(\{v\}, V_2)$.

Let $D$ be a directed graph. We use $V(D)$ and $A(D)$ to denote the vertex set and the arc set of $D$. If $uv\in A(D)$, then we say $u$ {\it dominates} $v$,  and denote it by $u \rightarrow v$. For two disjoint subsets $X$,  $Y$ of $V(D)$,  if each arc $uv$ between $X$ and $Y$ satisfies that $u\in X$ and $v\in Y$,  then we say $X \rightarrow Y $.   For a vertex $u\in V(D)$,  denote by $N^+_D(u)$ the set of vertices that are dominated by $u$,  and denote by $N^-_D(u)$ the set of vertices that are dominating $u$. The cardinality of $N^+_D(u)$ and $N^-_D(u)$ are denoted by $d^+_D(u)$ and $d^-_D(u)$, respectively. When this is no ambiguity, we often omit the subscript $D$. Let $P=v_1v_2\cdots v_{t+1}$ be a directed path in $D$ with $v_i \rightarrow v_{i+1}$ for $i=1,2,\ldots,t$. We say the length of $P$ is $t$ and the directed path $P$ is starting at $v_1$ and ending at $v_{t+1}$.

For other notations and terminologies not defined here, we refer the reader to \cite{bondy2008graph}.

\subsection{Extensible structures}
The main challenge for embedding PC spanning trees is to merge pieces of PC structures together. We find the following configurations can be used as glue for extending a PC path with certain constraints.

\begin{definition}[\textbf{Center edge of a colored triangle}]
Let $S$ be an edge-colored triangle, if
$col(e)\not\in col(S-e)$ for an edge $e$ in $S$, then we say $e$ is a \textit{center edge} of $S$.
\end{definition}
Apparently, every rainbow triangle contains $3$ center edges, every 2-colored triangle contains a unique center edge and every monochromatic triangle contains no center edge.

\begin{definition}[\textbf{Nice $t$-shovel}]
Let $Y$ be an edge-colored graph, consisting of a path $P=w_1w_2\cdots w_t~(t\geq 1)$ and a triangle $S=u_1u_2u_3u_1$ with $V(P)\cap V(S)=\{w_1\}=\{u_1\}$, as shown in the Figure \ref{fig:shovel}.  We call $Y$ a \textit{$t$-shovel}. If $P$ is a PC path, $u_2u_3$ is a center edge of $S$ and $col(w_1w_2)\not\in col\left(w_1,\{u_2,u_3\}\right)$, then we say $Y$ is a \textit{nice $t$-shovel}. For a nice $t$-shovel $Y$ with $t\geq 2$, we call $u_1=w_1$ the {\it center} of $Y$.
\end{definition}

\begin{definition}[\textbf{Nice bowtie}]
Let $F_1$ be a graph obtained by identifying one vertex of two disjoint triangles (see Figure~\ref{fig:short-B}) and let  $F_2$ be a graph obtained by adding an edge between two disjoint triangles (see Figure~\ref{fig:long-B}). Then we call both $F_1$ and $F_2$ \textit{bowtie}. To be specific, we say $F_1$ is a \textit{short bowtie} and $F_2$ is a \textit{long bowtie}. Let $F_1$ and $F_2$ be labeled as Figure~\ref{fig:short-B} and Figure~\ref{fig:long-B}. We say $u_0$ is the \textit{center} of $F_1$ and $u_3, u_4$ are centers of $F_2$. If $u_1u_2$ and $u_3u_4$ in $F_1$ are center edges of their lying triangles  and $col(u_0, \{u_1,u_2\})\cap col(u_0, \{u_3,u_4\})=\emptyset$, then we say $F_1$ is a \textit{nice short bowtie} (or nice bowtie).  If $F_2-\{v_1,v_2\}$ and $F_2-\{v_5,v_6\}$ are nice shovels, then we say $F_2$ is a \textit{nice long bowtie} (or nice bowtie).
\end{definition}

\begin{figure}[h]
    \centering
    \subfigure[{A $t$-shovel }]
    {
    \includegraphics[width=0.30\linewidth]{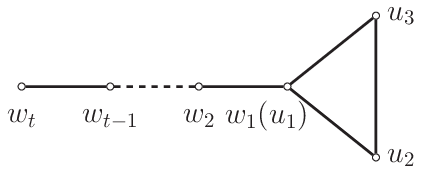}
    \label{fig:shovel}
     }
       \hskip 1.5cm
    \subfigure[{A short  bowtie $F_1$ with $u_0$ as a center}]
    {
     \includegraphics[width=0.18\linewidth]{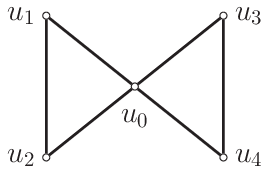}
     \label{fig:short-B}
    }
    \hskip 1.5cm
        \subfigure[{A long  bowtie $F_2$  with $v_3$ and $v_4$ as centers}]
     {
     \includegraphics[width=0.23\linewidth]{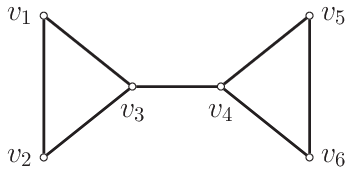}
      \label{fig:long-B}
    }
\caption{Extensible structures}
\end{figure}
The following lemmas are useful in the proofs of the main theorems.
\begin{lemma}\label{Lem:PCp}
Let $G$ be a mono-$C_3$-free edge-colored $K_n$ and let $P=v_1v_2\cdots v_t$ be a PC path in $G$ with $t\geq 2$.  If a vertex $v\in V(G)\setminus V(P)$ satisfies $col(vv_i)=col(v_iv_{i+1})$ for some $i\in [1,t-1]$, then there exists $j\in [i,t-1]$ such that $Q=v_1v_2\cdots v_jvv_{j+1}\cdots v_{t}$ is a PC path with $col(vv_j)=col(v_jv_{j+1})$.
 \end{lemma}
\begin{proof}
Let $j$ be the largest integer in $[1,t-1]$ satisfying $col(vv_j)=col(v_jv_{j+1})$. Then $j\geq i$. Since $vv_jv_{j+1}v$ is not a monochromatic $C_3$, we have $col(vv_j)\neq col(vv_{j+1})$. If $j=t-1$, then $Q=v_1v_2\cdots v_{t-1}vv_{t}$ is a desired PC path. If $j\leq t-2$, then by the maximality of $j$, we have $col(vv_{j+1})\neq col(v_{j+1}v_{j+2})$, which implies $Q=v_1v_2\cdots v_{j}vv_{j+1}\cdots v_t$ is a desired PC path.
\end{proof}
\begin{lemma}\label{Lem:STAR}
    Let $G$ be a mono-$C_3$-free edge-colored $K_n$. Then for each vertex $v\in V(G)$, there is a PC Hamilton path in $G$ starting at $v$.
\end{lemma}
\begin{proof}
  Fix a vertex $v$, let $P=v_1v_2\cdots v_t$ be a longest PC path in $G$ with $v_t=v$. Suppose, to the contrary, that $P$ is not a Hamilton path. Then for a vertex $u \in V(G)\setminus V(P)$, there holds $col(uv_1)=col(v_1v_2)$. Apply Lemma \ref{Lem:PCp} to the pair $(P,u)$. We get a longer PC path starting at $v_t=v$, a contradiction.
\end{proof}

\begin{lemma}[Spanning nice shovel]\label{Lem:SPANING}
     Let $G$ be a mono-$C_3$-free edge-colored $K_n$ with $n\geq 3$. Then $G$ contains a spanning subgraph which is a nice $(n-2)$-shovel.
\end{lemma}
\begin{proof}
    When $n=3$, it is trivial.
    When $n=4$, if $G$ has a vertex $v$ satisfying $d^c(v)=2$ or $G$ is a PC graph, then we can obtain a nice shovel immediately. Otherwise since no monochromatic triangle exists, $G$ has only one vertex of color degree 1 and three vertices of color degree 3, which also implies the existence of a spanning nice shovel.

    Now consider the case $n\geq5$. Let $Y$ be a largest PC shovel contained in $G$ with $S=u_1u_2u_3u_1$ being the triangle in $Y$, $u_1=w_1$ and $P=  w_1w_2\cdots w_t$ being the remaining path in $Y$. The above analysis implies $t\geq 2$. Suppose that $V(Y)\neq V(G)$. Let $v$ be a vertex in $V(G)\setminus V(Y)$.

    \begin{case}
         $col(vu_1)\not\in col(u_1,\{u_2,u_3\})$.
    \end{case}
    For each $j\in [1,t-1]$, define $P_j=w_1w_2\cdots w_jvw_{j+1}\cdots w_t$ and define $P_t=w_1w_2\cdots w_tv$. If $P_j$ is a PC path for some $j\in [2,t]$, then $P_j$ and the triangle $u_1u_2u_3u_1$ form a nice shovel larger than $Y$, a contradiction. So none of the paths $P_t,P_{t-1},\ldots, P_2$ is a PC path. This implies that $col(vw_j)=col(w_jw_{j-1})$ for all $j\in [2,t]$. Since the triangle $vw_2w_1v$ is not monochromatic we have $col(vw_2)\neq col(vw_1)$.  Therefore $P_1=w_1vw_2\cdots w_t$ and the triangle $u_1u_2u_3u_1$ form a nice shovel larger than $Y$, a contradiction.

    \begin{case}
    $col(vu_1)\in col(u_1,\{u_2,u_3\})$.
    \end{case}
    Without loss of generality, assume that $col(vu_1)=col(u_1u_2)=\alpha$. For each $j\in [2,t-1]$, define $P_j=w_1w_2\cdots w_ju_3w_{j+1}\cdots w_t$ and define $P_t=w_1w_2\cdots w_tu_3$. If $P_j$ is a PC path for some $j\in [2,t]$, then $P_j$ and the triangle $u_1u_2vu_1$ form a nice shovel larger than $Y$, a contradiction. So none of the paths $P_t,P_{t-1},\ldots, P_2$ is a PC path. This implies that $col(u_3w_j)=col(w_jw_{j-1})$ for all $j\in [2,t]$.
    Since the triangle $u_3w_2w_1u_3$ is not monochromatic we have $col(u_3w_2)\neq col(u_3w_1)$.  If $col(u_3w_1)\neq \alpha$, then  $P_1=w_1u_3w_2\cdots w_t$ and the triangle $u_1u_2vu_1$ form a nice shovel larger than $Y$(see Figure \ref{fig:lem2-1}), a contradiction. Therefore $col(u_3w_1)=col(u_2w_1)=col(vw_1)=\alpha$ and $\alpha \neq col(w_1w_2)$. Now consider the triangle $vu_2u_3v$. None of the three edges in the triangle is of color $\alpha$ (otherwise together with $u_1$, we get a monochromatic triangle). Without loss of generality, assume $u_2u_3$ is a center edge of the triangle $vu_2u_3v$. Then $P_1=vw_1w_2\cdots w_t$ and the triangle $vu_2u_3v$ form a nice shovel larger than $Y$(see Figure \ref{fig:lem2-2}), a contradiction.
    \begin{figure}[h]
    \centering
    \subfigure[{$col(u_3u_1)\neq \alpha$}]
    {
    \includegraphics[width=0.3\linewidth]{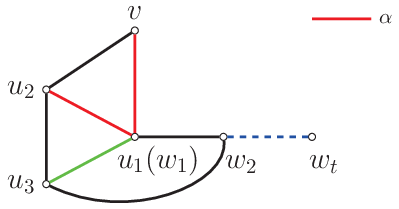}
    \label{fig:lem2-1}
     }
       \hskip 1.5cm
    \subfigure[{$col(u_3u_1)=\alpha$}]
    {
     \includegraphics[width=0.33\linewidth]{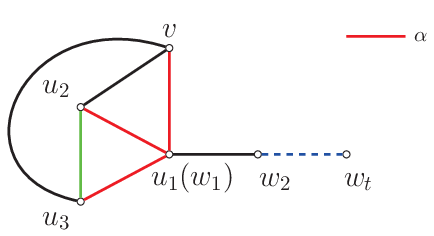}
     \label{fig:lem2-2}
    }
\caption{Cases for a larger nice shovel}
\end{figure}
\end{proof}

\begin{lemma}\label{Lem:S AND Y}
Let $G$ be a mono-$C_3$-free edge-colored $K_n$. Let $S$ and $Y$ be vertex-disjoint subgraphs of $G$ such that $S=u_1u_2u_3u_1$ is a triangle with $u_2u_3$ being a center edge and $Y$ is a nice shovel or a vertex set of size at most 2. Then $G$ contains a PC path $P=z_1z_2\cdots z_\ell$ such that $V(P)=V(S)\cup V(Y)$, $z_1=u_1$ and $col(z_1z_2)\in col(u_1,\{u_2,u_3\})$.
\end{lemma}
\begin{proof}
We leave the easy case $|Y|\leq 2$ to be checked by readers. Now assume that $Y$ is a nice shovel. Let $S^{\prime}  =v_1v_2v_3v_1$ be the unique triangle in $Y$ with $v_2v_3$ being the center edge. Let $w_1=v_1$ and $w_1w_2\cdots w_t$ be the unique Hamilton path in $Y-\{v_2, v_3\}$. Let $col(u_2u_3)=c_1$ and $col(v_2v_3)=c_2$. If there is an edge between $\{u_2,u_3\}$ and $\{v_2,v_3\}$ assigned a color distinct to both $c_1$ and $c_2$, then a desired PC path can be found immediately. This argument, together with the mono-$C_3$-free condition, implies that $col(u_2u_3)\neq col(v_2v_3)$.
Without loss of generality, assume $col(u_2u_3)=1$, $col(v_2v_3)=2$ and $col(u_2v_2)=1$. The mono-$C_3$-free condition again forces $col(u_3v_2)=2$, $col(u_3v_3)=1$ and $col(u_2v_3)=2$ (see Figure \ref{fig:1122}).
\begin{figure}[h]
    \centering
    \includegraphics[width=0.55\linewidth]{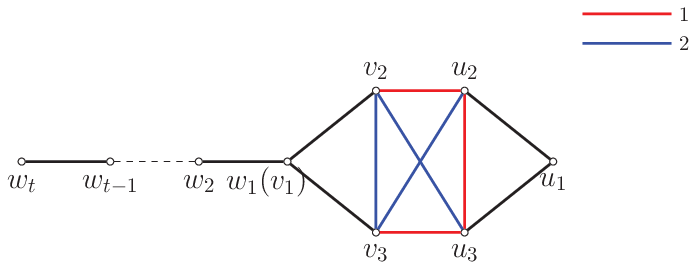}
     \caption{Colors between $S$ and $Y$}
     \label{fig:1122}
\end{figure}

If $col(u_1u_2)\ne 2$ or $col(v_1v_3)\ne 1$,  then
 $u_1u_2v_3u_3v_2v_1w_2\cdots w_t$ or $u_1u_2v_2u_3v_3v_1w_2\cdots w_t$ is a desired  PC path. For the remaining case that $col(u_1u_2)= 2$ and $col(v_1v_3)= 1$, we consider the PC path $P_1=v_3v_1w_2\cdots w_t$. Since $col(u_3v_3)=col(v_3v_1)=1$,  by Lemma \ref{Lem:PCp},  $G$ contains a PC  path $P^{\prime}  _1=x_1x_2\cdots x_{t+2}$,  where $V(P_1^{\prime})=V(P_1)\cup\{u_3\}$, $x_1=v_3$ and  $col(x_1x_2)=col(v_3v_1)$.  Then  $u_1u_2v_2x_1x_2\cdots x_{t+2}$ is a desired  PC path. This complets the proof.
\end{proof}

\section{Directed almost spanning trees in mono-$C_3$-free tournament}\label{sec:3}
Let $G$ be an edge-colored graph. If there is a function $h:V(G)\rightarrow col(G)$ such that $col(uv)=h(u)$ or $h(v)$ for each edge $uv\in E(G)$. Then we can define a directed graph $D_G$ with $V(D_G)=V(G)$ and $uv\in A(D_G)$ if and only if $col(uv)=h(u)$ and $col(uv)\neq h(v)$. Such an edge-colored graph $G$ is called ``degenerate'' in \cite{li_classification_2020}. An easy observation in \cite{li_vertexdisjoint_2020} states that each PC cycle in $G$ must be a directed cycle in $D_G$ and vise versa.
This one-to-one relationship does not hold if we study paths or trees instead of cycles. One can easily check that a PC path in $G$ is not necessarily a directed path in $D_G$. However, certain oriented trees in $D_G$ correspond to PC trees in $G$.

Let $G$ be a mono-$C_3$-free edge-colored $K_n$. When $G$ is degenerate, we will see that $D_G$ acts like a tournament. In the following, we define such a digraph and obtain an oriented-tree-embedding result by several lemmas. This embedding result is applied in the proof of Theorem \ref{Thm:k-spider}.

\begin{definition}\label{def:monoC3-free-tournament}
    Given $t\geq 2$, let  $D$ be a multipartite tournament such that each partite set has at most 2 vertices and $N^+\left( x \right) \cap N^+\left( y \right) =\emptyset$ for every pair of distinct vertices $x$ and $y$ from a same partite set. Then we say $D$ is a {\it mono-$C_3$-free tournament}.
\end{definition}

\begin{lemma}[Li \cite{li_properly_2021}]\label{Lem:D-cycle}
Given $t\geq2$, let  $D$ be a strongly connected mono-$C_3$-free tournament with $|V(D)|\ge 4$,  then each vertex of $D$ is contained in directed cycles of lengths from $4$ to $|V(D)|$.
\end{lemma}
\begin{lemma}\label{Lem:D-cycle+}

Given $t\geq2$, let  $D$ be a mono-$C_3$-free tournament. Let $D_0,D_1,\ldots,D_p$ ($p\geq 0$) be  strongly  connected components of $D$ such that for $0\leq i<j\leq p$, either $D_j\rightarrow D_i$ or there is no arc between $D_i$ and $D_j$. If $|V(D)|\ge 3$, then the following statements hold.\\
 $(a)$  If $p=0$, then $D$ contains a directed Hamilton cycle.\\
 $(b)$  If $p\geq 1$  and $|N^+_D(v)|\geq 1$ for all $v\in V(D)$, then each partite set of size $2$ is contained in $D_0$, $D-D_0$ is a tournament and $D_j\rightarrow D_i$ for $0\leq i<j\leq p$.\\
 $(c)$ If $|N^+_D(v)|\geq 1$ for all $v\in V(D)$, then $D$ contains a directed Hamilton path.\\
 $(d)$ If there are at least two partite sets of size $2$ in $D$, then all the partite sets of size $2$ are contained in $D_0$, $D-D_0$ is a tournament and $D_j\rightarrow D_i$ for $0\leq i<j\leq p$.
\end{lemma}

\begin{proof}
    The case $|V(D)|=3$ can be verified immediately. Assume $|V(D)|\geq 4$. Then the statement $(a)$ holds by Lemma \ref{Lem:D-cycle}. To prove $(b)$, suppose that some partite set  $W=\{x,y\}$ is not contained in $D_0$. If $W\cap V(D_0)=\emptyset$, then each $z\in V(D_0)$ must be a common out-neighbour of $x$ and $y$, a contradiction. If $W\cap V(D_0)\neq \emptyset$, say $W\cap V(D_0)=\{x\}$ and $y\in D_i$ for some $i\in[1,p]$, then by the condition that $|N^+_D(x)|\geq 1$, there must be a vertex $z\in V(D_0)$ such that $x\rightarrow z$. Note that $D$ is a multipartite tournament, we have $y\rightarrow z$, a contradiction. Therefore each partite set of size $2$ is contained in $D_0$, which implies that $D-D_0$ is a tournament and $D_j\rightarrow D_i$ for $0\leq i<j\leq p$. So the statement $(b)$ holds. The statement $(c)$ can be obtained directly by statements $(a)$ and $(b)$.   To prove $(d)$, assume that  $W_1,W_2,\ldots,W_s$ are partite sets of size $2$ in $D$ with $s\geq2$ and  $W_i=\{x_i, y_i\}$ $(i\in[1,s])$. For distinct indices $a$ and $b$ in $[1,s]$, let $D^\prime=D[W_a\cup W_b]$. If there is a vertex  $v\in V(D^{\prime})$ such that $d^-_{D^{\prime}}(v)=2$, then either $v\in N^+(x_a) \cap N^+(y_a)$ or $v\in N^+(x_b) \cap N^+(y_b)$, a contradiction. So $D^{\prime}$ must be a directed cycle of length $4$. By the arbitrariness of $a$ and $b$,  we know that $D[\cup^s_{i=1}W_i]$ is a strongly connected digraph.  Therefore $D[\cup^s_{i=1}W_i]\subseteq D_0$ (otherwise each vertex $v\in V(D_0)$ is in $ N^+(x_i) \cap N^+(y_i)$ for every $W_i$),  which implies that $D-D_0$ is a tournament and $D_j\rightarrow D_i$ for $0\leq i<j\leq p$.  So the statement $(d)$ holds.
\end{proof}
\begin{lemma}\label{Lem:D-path}
    Let $D$ be a mono-$C_3$-free tournament. Let  ${P}=v_1v_2\cdots v_t $ be a directed path in $D$. Assume  there exists  a vertex $v\in V(D)\setminus V({P})$  such that   $v\rightarrow v_i$ for some $i\in[1, t]$.  Then $D$ contains a directed path ${P^{\prime}  }=u_1u_2\cdots u_{t+1}$,  where $V({P^{\prime}  })=V({P})\cup\{v \}$ and  $u_{t+1}=v_t$.
\end{lemma}
    \begin{proof}
        Let $j$ be a minimum integer in $[1,t]$ such that $v\rightarrow v_j$. If $j=1$, then ${P^{\prime}  }=vv_1\cdots v_{t-1}v_t$ is a desired directed path. If $j\geq 2$, then
        $v_j\in N^+(v)\cap N^+(v_{j-1})$. This implies that $v$ and $v_{j-1}$ belong to distinct partite set. By the assumption of $j$, we have $vv_{j-1}\not\in A(D)$. Therefore $v_{j-1}\rightarrow v$ and ${P^{\prime}  }=v_1 \cdots v_{j-1}vv_jv_{j+1}\cdots v_t$ is a directed path in $D$.
    \end{proof}

Given positive integers $\{l_i\}^k_{i=2}(k\geq 2)$,  let $P_i$ be a directed path of length $l_i$ from $u_i$ to $u$  for each $i\in[2,k]$, satisfying  $V(P_i)\cap V(P_j)=\{u\} $  for $2\leq i<j\leq k$. Let $x,y$ be two vertices not in $\cup ^k_{i=2}V(P_i)$. Define a directed tree $T^*_{ l_2, \ldots, l_k}$ (see Figure \ref{fig:d-tree}) as
$$V(T^*_{ l_2, \ldots, l_k})=\{x,y\}\cup\bigcup_{i=2}^k V(P_i)\text{~~and~~} A(T^*_{ l_2, \ldots, l_k})=\{ux,uy\}\cup\bigcup_{i=2}^k A(P_i). $$ We call $u$ the {\it root} of $T^*_{ l_2, \ldots, l_k}$.

\begin{figure}[h]
   \centering
    \includegraphics[width=0.25\linewidth]{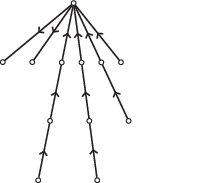}
    \caption{The tree $T^*_{3, 3, 2,1}$}
    \label{fig:d-tree}
\end{figure}

\begin{theorem}\label{Thm:D-spider}
  For $k\ge 2$ and $\ell_2\geq \ell_3 \geq \cdots \geq \ell_k\geq 1$,  let $D$ be a mono-$C_3$-free tournament on $n\geq\sum^k_{i=2}{\ell_i}+2k^2+2k+6$ vertices. If $|N^+_D(v)|\geq 2$  for all $v\in V(D)$.  Then $D$ contains  a  subgraph isomorphic to   $T^*_{\ell_2,\ldots,\ell_k}$.
\end{theorem}

\begin{proof}
    We use  $D_0,D_1\ldots,D_p$ to denote the strongly connected components  of $D$. Let $U=V(D)\setminus V(D_0)$. If $|U|\geq \sum_{i=2}^k{\ell_i}$, then $p\geq 1$ and by Lemma \ref{Lem:D-cycle+}(b), we can assume that $D_j\rightarrow D_i$ for $0\le i < j\le p$ and $D[U]$ is a tournament with $P$ being a directed Hamilton path in $D[U]$. Decompose $P$ into $k-1$ vertices disjoint directed paths $P_2,P_3,\ldots,P_k$ ending at $v_2,v_3,\ldots,v_k$ of length  $\ell_2-1,\ell_3-1, \ldots, \ell_k -1$, respectively.
    Take a vertex $v_0\in D_0$. Then $v_i\rightarrow v_0$ for all $2\leq i\leq k$.
    Since $|N^+(v_0)|>2$,  there are two distinct vertices $x,y\in D_0$ satisfying $v_0\rightarrow x$ and $v_0\rightarrow y$. Those paths and together with $x,y,v_0$ form a tree $T^*_{\ell_2, \ldots, \ell_k}$.
    The remaining case is  $|U|<\sum_{i=2}^k{\ell_i}$. Take $k-1$ integers $\ell^\prime_2, \ldots, \ell^\prime_k$  such that $0\leq \ell^\prime_i\le \ell_i$ for each $i\in[2,k]$ and $\sum_{i=2}^k{\ell^\prime_i}+|U|= \sum_{i=2}^k{\ell_i}$. Then $|V(D_0)|\geq\sum_{i=2}^k{\ell^\prime_i}+2k^2+2k+6$.
   \begin{claim}\label{claim:D-2}
        If $D_0$ contains a  subgraph $T_1$ isomorphic to   $T^*_{ \ell^{\prime}  _2, \ldots, \ell'_k}$, then $D$ contains  a   subgraph isomorphic to   $T^*_ {\ell_2, \ldots, \ell_k}$.
   \end{claim}
 \begin{proof}
  Denote by $u$ the root of $T_1$ and by $u_2,\ldots,u_k$  the $k$ leafs of $T_1$ with $d^-_{T_1}(u_i)=0$ for $2\leq i\leq k$.    Define $r_i = \ell_i - \ell'_i$ for $i \in [2, k]$.  Since $|U|=\sum_{i=2}^k r_i$ and $D[U]$ is a tournament (by Lemma \ref{Lem:D-cycle+}(b)), $U$ can be partitioned into $k-1$ parts such that the $i$-th part is a tournament of order $r_i$. Note that for every pair of vertices $u\in U$ and $v\in D_0$, we have $u\rightarrow v$. Hence $T_1$ can be extended into a tree $T^*_ {  \ell_2, \ldots, \ell_k}$ by vertices in $U$.
 \end{proof}

By Claim \ref{claim:D-2}, it remains to prove that $D_0$ contains a  $T^*_{ \ell^{\prime}_2, \ldots, \ell^{\prime}  _k}$ as a subgraph. Without loss of generality, assume that $\ell^{\prime}_i\geq 1$ for $i\in [2,t]$.
Let $v$ be a vertex with minimum out-degree in $D_0$. Let $A=N^+_{D_0}(v)$ and $B=N^-_{D_0}(v)$. Then $|A|\leq \frac{ |D_0|-1}{2}$. Since each vertex in $D_0$ is adjacent to at least $|D_0|-2$ vertices, we have
\begin{equation}\label{bound:B}
    |B|\geq |D_0|-2-|A|\geq |D_0|-2-\frac{ |D_0|-1}{2}\geq \frac{ |D_0|-3}{2}\geq k^2+k+1.
\end{equation}

If $|B|\geq \sum_{i=2}^k \ell^\prime_i$,  then let $P$ be a longest directed path in $D_0[B\cup\{v\}]$ with $v$ being the ending vertex. By Lemma~\ref{Lem:D-path}, $P$ is a directed Hamilton path in $D_0[B\cup\{v\}]$.  Decompose $P$ into disjoint directed paths $P_2,\ldots,P_k$ ending at $v_2,v_3,\ldots,v_k$ of lengths  $\ell^{\prime}  _2-1,\ell^{\prime}  _3-1, \ldots, \ell^{\prime}  _k -1$, respectively. Recall that $|N^+_D(v)|\geq 2$ and $N^+_D(v)=N^+_{D_0}(v)=A$. There are two distinct vertices $x,y \in A=N^+(v)$.  Note that $\{v_2,v_3,\ldots,v_k\}\subseteq N^-(v)$. We get a $T^*_{ \ell^{\prime}  _2, \ldots, \ell^{\prime}  _k}$ as $ (\cup_{i\in[2,k]}(P_i+v_iv))\cup \{uw_1, uw_2\}$.

Now we assume $|B|\leq \sum_{i=2}^k \ell^\prime_i-1$. This implies
\begin{equation}\label{bound:A}
    |A|\geq |D_0|-2-|B|\geq 2k^2+2k+4+\sum_{i=2}^k \ell^\prime_i-|B|\geq 2k^2+2k+5.
\end{equation}
\begin{claim}\label{cl:merge}
If $D_0[A]$ contains $k-1$ disjoint directed paths $P_2,P_3,\ldots,P_k$ ending at vertices $a_2,a_3,\ldots a_k$ respectively and satisfying the following conditions:  \\
$(i)$ $1\leq |P_i|\leq \max\{\ell_i^\prime-1,1\}$ and $d^+_B(a_i)\geq k-1$ for $2\leq i\leq k$;\\
$(ii)$ there is a set $W\subseteq A\setminus \cup_{i=1}^k V(P_i)$ with $w\rightarrow z$ for every $w\in W$ and every $z\in \cup_{i=1}^k V(P_i)$ or $W=\emptyset$ such that
$$
|W|+\sum_{i=2}^k|P_i|+|B|\geq \sum_{i=2}^k \max\{\ell^\prime_i,2\},
$$
then $D_0$ contains a $T^*_{ q_2, \ldots, q_k}$. Here $q_i=\max\{\ell^\prime_i,2\}$ for $2\leq i\leq k$.
\end{claim}
\begin{proof} Condition $(i)$ tells the existence of a $k-1$ matching $\{a_ib_i: 2\leq i\leq k\}$ for a set $\{b_i:2\leq i\leq k\}\subseteq B$. Let $u_i$ be the starting vertex of $P_i$ and let $L_i=u_iP_ia_ib_iv$ for $2\leq i\leq k$. Apparently $L_i$ is a directed path of length between $2$ and $\max\{\ell^\prime_i,2\}$. We assert that for each vertex $x\in (W\cup B)\setminus\{b_i: 2\leq i\leq k\}$ and each $L_i$, there is a vertex $y^x_i\in V(L_i)$ such that $x\rightarrow y^x_i$. This can be done by setting $y^x_i=v$ when $x\in B\setminus\{b_i: 2\leq i\leq k\}$ and setting $y^x_i=a_i$ when $x\in W$. Note that $|W|+\sum_{i=2}^k|P_i|+|B|\geq \sum_{i=2}^k \max\{\ell^\prime_i,2\}.$
By Lemma \ref{Lem:D-path}, we can extend each path $L_i$ first by vertices in $B\setminus\{b_i: 2\leq i\leq k\}$ and then by vertices in $W$ until each leg reach a length of $\max\{\ell^\prime_i,2\}$. Recall that $|A|\geq 2k^2+2k+4+\sum_{i=2}^k \ell^\prime_i-|B|$. There are at least $2k^2$ vertices in $A$ not added into legs. Choose two of them as out-neighbors of $v$. Then we find a $T^*_{ q_2, \ldots, q_k}$.
\end{proof}
To satisfy the condition of Claim \ref{cl:merge}, we need to analyse the structure of $D_0[A]$. If $D_0[A]$ is a tournament or $D_0[A]$ has at least two partite sets of size two, then set $A^\prime=A$. If there is exactly one partite set of size two contained in $D_0[A]$,  say $W=\{x,y\}$, then set $A^{\prime}=A-x$.
In both cases, we can assume that $Q_1,Q_2,\ldots,Q_t (t\geq 1)$ are strongly connected components of $D_0[A^\prime]$ with  $Q_j\rightarrow Q_i$ for $1\leq i<j\leq t$ (by Lemma \ref{Lem:D-cycle+}), $A^\prime\subset A$ and $|A^\prime|\geq |A|-1\geq 2k^2+2k+4$. Now we analyse $D_0[A^\prime]$ instead of $D_0[A]$. The following claim tells that either some $Q_j$ is large or there is an index $j$ such that both $\sum_{i\leq j} |Q_i|$ and $\sum_{i>j} |Q_i|$ are large.

\begin{claim}\label{Cl:two-cases}
   There exists some $j\in[1,t]$ such that one of the following holds\\
   $(i)$ $\sum_{i=1}^j |Q_i|\geq k-1$ and $\sum_{i=j+1}^t |Q_i|\geq k-1$;\\
   $(ii)$ $|Q_j|\geq |A^\prime|-2k+4$;
\end{claim}
\begin{proof}
   Recall that $|A^\prime|\geq 2k^2+2k+4$. Let $j$ be the smallest index in $[1,t]$ such that $\sum_{i=1}^j |Q_i|\geq k-1$. So $\sum_{i=1}^{j-1} |Q_i|\leq k-2$. If $(i)$ does not hold, then  $\sum_{i=j+1}^t |Q_i|\leq k-2$. Therefore
   \[|Q_j|\geq |A^\prime|-\sum_{i=j+1}^t |Q_i|-\sum_{i=1}^{j-1} |Q_i|\geq |A^\prime|-(k-2)-(k-2)=|A^\prime|-2k+4.\]
\end{proof}
\setcounter{case}{0}
Now we proceed the proof by analyzing the two cases in  Claim \ref{Cl:two-cases}.
\begin{case}
   There exists some $j\in[1,t]$ such that $\sum_{i=1}^j |Q_i|\geq k-1$ and $\sum_{i=j+1}^t |Q_i|\geq k-1$
\end{case}
In this case, apply Lemma \ref{Lem:D-cycle+}(c) to $D_0[\cup_{i=1}^j Q_i]$, we get a directed Hamilton path $P$ in $D_0[\cup_{i=1}^j Q_i]$.  Note that $N^+_{D_0}(z)\subseteq B\cup\left(\cup_{i=1}^j Q_i\right)=B\cup V(P)$ for each vertex $z\in V(P)$ and $v$ is a vertex in $D_0$ with minimum out-degree. We have
$$d^+_{D_0}(z,B)\geq d^+_{D_0}(z)-|P|\geq d^+_{D_0}(v)-|P|=|A|-\sum_{i=1}^j |Q_i|\geq \sum_{i=j+1}^t |Q_i|\geq k-1.$$
Set $W:=\cup_{i=j+1}^t Q_i.$
If $|P|\geq \sum_{i=2}^k \max\{\ell^\prime_i-1,1\}$, then find $k-1$ disjoint paths $P_2,P_2,\ldots,P_k$ in $P$ such that $|P_i|= \max\{\ell^\prime_i-1,1\}$ for $i\in [2,k]$. Then
$$|W|+\sum_{i=2}^k|P_i|+|B|>\sum_{i=2}^k|P_i|+k-1\geq \sum_{i=2}^k \max\{\ell^\prime_i,2\}.$$

If $|P|<\sum_{i=2}^k \max\{\ell^\prime_i-1,1\}$, then let $P_2,P_3,\ldots,P_k$ be a decomposition of $P$ into $k-1$ disjoint paths such that $1\leq |P_i|\leq \max\{\ell^\prime_i-1,1\}$ for $i\in [2,k]$ (this is possible since $|P|=\sum_{i=1}^j |Q_i|\geq k-1$).
By (\ref{bound:A}), we have
$$|W|+\sum_{i=2}^k|P_i|+|B|=|A^\prime|+|B|\geq \sum_{i=2}^k \ell^\prime_i+2k^2+2k+3>\sum_{i=2}^k \max\{\ell^\prime_i,2\}.$$
In both situations, we obtain a tree $T^*_{ \ell^\prime_2, \ldots, \ell^\prime_k}$ by Claim \ref{cl:merge}.

\begin{case}
There exists some $j\in[1,t]$ such that $|Q_j|\geq |A^\prime|-2k+4$.
\end{case}
In this case, by Lemma \ref{Lem:D-cycle}, $Q_j$ contains a directed Hamilton cycle $C=w_1w_2\cdots w_hw_1$.
Let $h_2,h_3,\ldots,h_k$ be positive integers such that $h_i\leq \max\{\ell^\prime_i-1,1\}$ for $2\leq i\leq k$, $\sum_{i=2}^k h_i\leq |C|=h$  and $\sum_{i=2}^k h_i$ is maximum.
We say a sequence of indices $(\alpha_2,\alpha_3,\ldots,\alpha_k)$ is a {\it model} on $C$ if $\alpha_{j+1}=\alpha_{j}+h_{j}$ for $2\leq j\leq k-1$. Here the addition is on $\mathbb{Z}_h$. Given a model, we can find vertex disjoint paths $P_2,P_3,\ldots,P_k$ such that $P_i$ is of length $h_i-1$ and ending at the vertex $w_{\alpha_i}$.
Note that $\sum_{i=2}^k|P_i|\geq \min\left\{h, \sum_{i=2}^k(\ell^\prime_i-1)\right\}$. We have
\begin{align*}
\sum_{i=2}^k|P_i|+|B|&\geq \min\left\{h, \sum_{i=2}^k(\ell^\prime_i-1)\right\}+|B|\\
&\geq \min\left\{|A^\prime|-2k+4,\sum_{i=2}^k(\ell^\prime_i-1)\right\}+|B|\\
&\geq \min\left\{|A|+|B|-2k+3,\sum_{i=2}^k(\ell^\prime_i-1)+|B|\right\}.
\end{align*}
By (\ref{bound:B}) and (\ref{bound:A}), there holds
$\sum_{i=2}^k|P_i|+|B|\geq \sum_{i=2}^k \max\{\ell^\prime_i,2\}$.
If there is a model $(\alpha_2,\alpha_3,\ldots,\alpha_k)$ such that $d^+_{D_0}(w_{\alpha_i},B)\geq k-1$ for $2\leq i\leq k$. We can obtain a $T^*_{ \ell^\prime_2, \ldots, \ell^\prime_k}$ by Claim \ref{cl:merge}.

Now assume that for every model $(\alpha_2,\alpha_3,\ldots,\alpha_k)$, there exists a vertex $w_{\alpha_i}$ such that $d^+_{D_0}(w_{\alpha_i},B)\leq k-2$.
Let $M_0=(\alpha_2,\alpha_3,\ldots,\alpha_k)$ be a model on $C$. Define $M_j=(\alpha_2+j,\alpha_3+j,\ldots,\alpha_k+j)$ for $1\leq j\leq h-1$. Here the addition is on $\mathbb{Z}_h$. Then $M_j$ is also a model on $C$. Let $\mathcal{M}=\{M_i:0\leq i\leq h-1\}$. Let
$$X=\{\alpha: d^+_{D_0}(w_\alpha, B)\leq k-2 \}.$$
Note that each $\alpha\in X$ lies in exactly $k-1$ models in $\mathcal{M}$, and for each model $M_i$, there is at least one entry belonging to $X$. Therefore
$$|X|\geq \frac{|\mathcal{M}|}{k-1}=\frac{h}{k-1}\geq \frac{|A^\prime|-2k+4}{k-1}.$$
Let $H$ be the induced subgraph of $D_0$ on $\{w_\alpha:\alpha\in X\}$.  Then there is a vertex $z$ in $H$ with $d_H^+(z)\leq \frac{|X|-1}{2}$. Then
$$
\begin{aligned}
|A|=d^+_{D_0}(v)\leq d^+_{D_0}(z)&\leq d^+_{H}(z)+|A\setminus X|+d^+_{D_0}(z,B)\\
&\leq \frac{|X|-1}{2}+|A|-|X|+k-2.
\end{aligned}
$$
Hence $|X|\leq 2k-5$. Recall that $|X|\geq \frac{|A^\prime|-2k+1}{k-1}\geq \frac{|A|-2k}{k-1}$. We get
$$|A|\leq 2k^2-5k+5,$$
which contradicts (\ref{bound:A}) that $|A|\geq 2k^2+2k+5$. The proof is complete.
\end{proof}

\section{Proofs of Theorem \ref{Thm:subdivision of tree} and \ref{Thm:k-spider}}\label{sec:4}
Recall that $g(S_k,C_3)$ is the maximum number $N$ such that there exists an edge-colored $K_N$ containing neither a rainbow $S_k$ nor a monochromatic $C_3$. We will use this notation to state the proofs of main theorems.
 \begin{prop}\label{prop:gr}
     $g(S_k,C_3)<(k+1)!$ for every integer $k\geq 1$.
 \end{prop}
 \begin{proof}
 We prove by induction on $k$. When $k=1$, it is trivial. Assume that the statement holds for $g(S_{k-1},C_3)$. Let $G$ be a mono-$C_3$-free edge-colored $K_n$ with $n\ge (k+1)!$.
 By the inductive hypothesis, $G$ contains a rainbow $(k-1)$-star with $v$ being the vertex of degree $k-1$.  If    $d^c_G(v)\geq k$, then the proof is complete.  If $d^c_G(v)\leq k-1$,   then  there exists a  vertex set $W$ of size at least $\frac{(k+1)!-1}{k-1}\geq k!$.  Since $|W|\geq k!$, by the inductive hypothesis and the assumption that $G$ is mono-$C_3$-free, $G[W]$ contains a rainbow $(k-1)$-star $T$ with $u$ being the center vertex of degree $k-1$. Note that $G$ has no monochromatic triangle, we know that $col(v,W) \cap col(T)=\emptyset$. So $T+uv$ is a rainbow $k$-star.
 \end{proof}

  \begin{lemma} \label{lem:all PC tree}
      Given a positive integer $k\geq 1$, let $G$ be a mono-$C_3$-free edge-colored $K_n$ with $n\ge (k+2)! $, then $G$ contains a PC copy of every tree of $k$ edges.
  \end{lemma}

  \begin{proof}
By induction on $k$. The case $k=1$ can be verified immediately.  Assuming that the conclusion holds for every tree of $k-1$ edges. Let $T$ be a tree of $k$ edges with a leaf vertex $x$ and a pendent edge $xy$.
Define $V_0=\{v\in V(G): d^c(v)\leq k-1\}$ and $V_1=V(G)\setminus V_0$. Since $G[V_0]$ contains no rainbow $S_k$, by Proposition \ref{prop:gr}, we have  $|V_0|<(k+1)!$.  Therefore $|V_1|>(k+2)!-(k+1)! >(k+1)!.$ By the inductive hypothesis,  $G[V_1]$ contains a subgraph $T^{\prime}$ which is a PC copy of $T-x$. Let $\varphi:\ V(T)\setminus \{x\} \rightarrow V( T^{\prime})$ be an isomorphism mapping from $T-x$ to $T'$. Let $y^\prime=\varphi(y)$. By the definition of $V_1$, we have $d^c_G(y^\prime)\geq k$. Note that $|T^\prime -y|\leq k-1$. There exists a vertex $x^{\prime}  \in V(G)\setminus V(T^{\prime}  )$ such that $col(y^\prime x^{\prime}  )\notin col(y^\prime,V(T'))$. Hence $T^{\prime}  +y^{\prime}  x^{\prime} $   is a  PC copy of $T$ in $G$. The proof is complete.
 \end{proof}

\begin{proof}[\textbf{Proof of Theorem \ref{Thm:subdivision of tree}}]
By Lemma \ref{lem:all PC tree}, $G$ contains a PC copy of $T_0$. Denote by $T_1$ a maximal  PC subgraph of $G$ such that $T_1$ is a subdivision of $T_0$. If $V(T_1)=V(G)$, then we are  done. Otherwise, for each edge $xy\in E(T_1)$, denote by $T_1(x,y)$ and  $T_1(y,x)$  the two connected components of $T_1-xy$ containing $x$ and $y$, respectively. For a vertex $u \in V(G) \backslash V(T_1)$, if  $col(ux)=col(xy)$, then we say the triple $(u,x,y)$ is {\it compatible} to $T_1$. By the assumption that $T_1$ is maximal, for every leaf vertex $x$ and the pendent edge $xy$ in $T_1$, the triple $(u,x,y)$ is compatible. Now let $(u,x,y)$ be a triple compatible to $T_1$ with $|T_1(y,x)|$ being the minimum (here $x$ may not be a leaf vertex). Since $G$ is mono-$C_3$-free and $col(ux)=col(xy)$, we have $col(uy)\neq col(ux)$.
If $col(uy)\not\in col(y,N_{T_1(y,x)}(y))$, then $T_1-xy+xu+uy$ forms a PC subdivision of  $T_0$ larger than $T_1$,
a contradiction. Hence there exists a vertex $z\in N_{T_1(y,x)}(y)$ with $col(uy)=col(yz)$. This implies that  $(u,y,z)$ is compatible to $T_1$ with $|T_1(z,y)|<|T_1(y,x)|$, a contradiction.
\end{proof}

\setcounter{case}{0}
\begin{proof}[\textbf{Proof of Theorem \ref{Thm:k-spider}}]
The proof is given by distinguishing the number of vertex-disjoint nice bowties in $G$.
\begin{case}
    There are at least $g(S_k,C_3)+1$ vertex-disjoint nice bowties in $G$.
\end{case}
Let $B_1,\cdots,B_s$ be vertex-disjoint nice bowties in $G$ with $s\geq g(S_k,C_3)+1$ such that $G-\cup^h_{i=1} {V(B_i)}$  contains no nice bowtie.   Pick one center $v_i$ from each bowtie $B_i$ ($1\leq i\leq s$) to form a set  $W$.  Since $G$ is mono-$C_3$-free and $|W|\geq g(S_k,C_3)+1$, there exists a rainbow $k$-star $S_k$ in $G[W]$. Without loss of generality, assume $V(S_k)=\{v_0, v_1,v_2, \ldots, v_k\}$ and $v_0$ is the $k$-degree vertex in $S_k$.

For $1\leq i\leq k$, since $v_i$ is a center vertex of the bowtie $B_i$, there exists a subgraph $X_i\subseteq B_i$ such that $X_i+v_iv_0$ is either a nice 2-shovel or a nice 3-shovel. Denote by $Y_i$  the nice shovel $X_i+v_iv_0$. Then $Y_1,Y_2,\ldots,Y_k$ are disjoint shovels overlapping at the vertex $v_0$. We will obtain a desired PC spider by extending or shrinking these shovels.

Define $I=\{1,2,\ldots,k\}$ and $I^-=\{i\in I:  \ell_i \leq |Y_i|-1\}$.
For each $i\in I^-$,  we can modify $Y_i$ into a leg of length $\ell_i$ by removing  $|Y_i|-1-\ell_i $ vertices.  Let $R$  be the removed vertices. Then $U=(V(G)\setminus \cup^k_{i=1} V(Y_i))\cup R$ is the set of vertices waiting to be embedded. Note that $|U|=\sum_{i\in I\setminus I^-} {(\ell_i-|Y_i|+1)}$. We partition $U$ into $|I\setminus I^-|$ disjoint sets $\{U_i\}_{i\in I\setminus I^-}$  such that $|U_i|=\ell_i-|Y_i|+1$ for each $i\in I\setminus I^-$.
Then by Lemma \ref{Lem:SPANING}, each $G[U_i]$ of order at least $3$ contains a spanning nice shovel $F_i$. Set $F_i=U_i$ when $|U_i|\leq 2$. Apply Lemma \ref{Lem:S AND Y} to the triangle in $Y_i$ and $F_i$ for all $i\in I\setminus I^-$, we get a desired PC spanning tree in $G$.

\begin{case}
   The number of vertex-disjoint nice bowties in $G$ is at most $g(S_k,C_3)$.
\end{case}
In this case we will first find a structure that is almost the spider we want except the first leg, which is a nice 1-shovel. Then we get the final spider by extending this shovel into a leg of length $\ell_1$.
To state the proof, we need more definitions.
Let $S$ be a spider with precisely $k-1$ legs of lengths $\ell_2,\ell_3,\ldots,\ell_k$ respectively, let $Y$ be a triangle and let $O^*_{ \ell_2, \ldots, \ell_k}$ be the graph obtained by identifying a vertex in $Y$ with the center of $S$. We call $O^*_{ \ell_2, \ldots, \ell_k}$ an {\it octopus}. The legs of $O^*_{ \ell_2, \ldots, \ell_k}$ are exactly the legs of $S$. We say $O^*_{ \ell_2, \ldots, \ell_k}$ is {\it nice} if $S$ is a PC spider and $Y\cup P$ is a nice shovel for each leg $P$ of $S$. If $G$ contains a {nice} octopus $O^*_{ \ell_2, \ldots, \ell_k}$, then let $F=G-O^*_{ \ell_2, \ldots, \ell_k}$.
By Lemma \ref{Lem:SPANING},  $F$ contains a spanning nice shovel $F^\prime$ when $|F|\geq 3$. Set $F^\prime=F$ when $|F|\leq 2$. Apply Lemma \ref{Lem:S AND Y} to $F^\prime$ and the triangle in $O^*_{ \ell_2, \ldots, \ell_k}$, we get a desired PC spider. Therefore the following claim holds.
\begin{claimA}\label{Cl:octopus}
	If $G$ contains a nice octopus $O^*_{ \ell_2, \ldots, \ell_k}$, then $G$ contains a PC spider with legs of lengths $\ell_1,\ell_2, \ldots, \ell_k$.
\end{claimA}

Let $B_1,B_2,\ldots,B_s$ be vertex-disjoint nice bowties in $G$ such that there is no nice bowtie in $H:=G\setminus \cup_{i=1}^s V(B_i)$. Then $s\leq g(S_k,C_3)$. Recall that each bowtie has at most $6$ vertices. We have $|H|\geq n-6g(S_k,C_3)$. Note that $\ell_1\geq \frac{n-1}{k}\geq 6g(S_k,C_3)+2k^2+2k+7$. Then
$$|H|=1+\sum_{i=2}^k\ell_i+\ell_1-6g(S_k,C_3)\geq \sum_{i=2}^k\ell_i+2k^2+2k+8.$$

For each vertex $v$ in $H$, we claim that at most one color in $col(v,N_{H}(v))$ appears more than once at $v$. Otherwise by the mono-$C_3$-free condition, $H$  contains a nice short bowtie with $v$ being the center vertex, which contradicts the choice of $H$.

\begin{claimA}\label{Cl:color-once}
For each vertex $v\in V(H)$, at most one color in $col(v,N_{H}(v))$ appears more than once at $v$.
\end{claimA}

There could be a vertex $v$ with no color repeats at $v$. Let $V_1=\{v\in V(H): \Delta^{mon}_H(v)=1\}$ and $V_2=\{v\in V(H): \Delta^{mon}_H(v)\geq 2\}$.

If $|V_1|\geq 3$, then let $v_1,v_2,v_3$ be distinct vertices in $V_1$. By the definition of $V_1$, the cycle $v_1v_2v_3v_1$ must be a rainbow triangle. Let $H^\prime=H-\{v_1,v_2,v_3\}$. Recall that $|H|\geq \sum_{i=2}^k\ell_i+2k^2+2k+8$. We have $|H^\prime|\geq \sum_{i=2}^k\ell_i+2k^2+2k+5$. Let $X_2, X_3,\ldots, X_k$ be disjoint sets in $H^\prime$ such that $|X_i|=\ell_i$ for $2\leq i\leq k$. By Lemma \ref{Lem:STAR}, there is a PC path $P_i$ on $\{v_1\}\cup X_i$ with $v_1$ being the starting vertex for $2\leq i\leq k$. Since $\Delta^{mon}_{H^\prime}(v_1)=1$, the triangle $v_1v_2v_3v_1$ and paths $P_2,P_3,\ldots,P_k$ form a nice octopus $O^*_{ \ell_2, \ldots, \ell_k}$. By Claim \ref{Cl:octopus}, we are home.

If $2\leq \Delta^{mon}_{H}(x)\leq 2k^2+2k+7$ for some vertex $x\in V_2$, then assume $\alpha$ is the color such that $\Delta^{mon}_{H}(x) = |\{u\in V(H): col(ux) = \alpha\}|$.  Let $U_\alpha=\{u\in V(H): col(ux) = \alpha\}$ and $H^\prime=H-U_\alpha$. Then $|H^\prime|\geq \sum_{i=2}^k\ell_i+1$. Let $X_2, X_3,\ldots, X_k$ be disjoint sets in $H^\prime-x$ such that $|X_i|=\ell_i$ for $2\leq i\leq k$. By Lemma \ref{Lem:STAR}, there is a PC path $P_i$ on $\{x\}\cup X_i$ with $v_1$ being the starting vertices for $2\leq i\leq k$. Note that $\Delta^{mon}_{H^\prime}(x)=1$. Choose distinct vertices $a,b\in U_\alpha$, then the triangle $xabx$ and paths $P_2,P_3,\ldots,P_k$ form a nice octopus $O^*_{ \ell_2, \ldots, \ell_k}$. By Claim \ref{Cl:octopus}, we are home.

The remaining case is that $|V_1|\leq 2$ and $\Delta^{mon}_{H}(x)\geq 2k^2+2k+8$ for every $x\in V_2$. Let $H_1=H-V_1$.
Then $|H_1|\geq \sum_{i=2}^k\ell_i+2k^2+2k+6$ and $\Delta^{mon}_{H_1}(u)\geq 2k^2+2k+6$ for every vertex $u\in V(H_1)$. Let $f(u)$ be the unique color appearing more than once in $col(u,N_{H_1}(u))$. We claim that for any two vertices $u, v\in V(H_1)$,  there holds $col(u)=f(u)$ or $f(v)$. Otherwise, $H$ contains a nice long bowtie with $u$ and $v$ being the center vertices, which contradicts the choice of $H$.
\begin{claimA}\label{Cl:degenerate}
For each pair of vertices $u,v\in V(H)$, we have $col(uv)=f(u)$ or $f(v)$.
\end{claimA}
Now construct an auxiliary digraph $D$ satisfying $V(D)=V(H_1)$ and $A(D)=\{uv: col(uv)=f(u) \ \text{and}\  col(uv)\ne f(v) \}.$ Since $H_1$ is mono-$C_3$-free, according to Claim \ref{Cl:degenerate} and Definition \ref{def:monoC3-free-tournament}, the directed graph $D$ is a mono-$C_3$-free tournament with $N^+_{D}(u)\geq 2k^2+2k+6>2$ for every vertex $u\in V(H_1)$. Recall that $|H_1|\geq \sum_{i=2}^k\ell_i+2k^2+2k+6$.
By Theorem \ref{Thm:D-spider}, $D$ contains a subgraph $ T^*_{\ell_2, \cdots, \ell_k}$, which is a nice octopus $O^*_{ \ell_2,\ldots,\ell_k}$ in $G$. Again we obtain the desired PC spanning tree by Claim \ref{Cl:octopus}. The proof is complete.
\end{proof}

\section{Conclusion}\label{sec:conclusion}
Tree embedding has become an increasingly popular research area in recent years. It has been extensively studied in various contexts, such as spanning trees or almost spanning trees in random graphs \cite{montgomery_spanning_2019}, random directed graphs \cite{montgomery_spanning_nodate}, $(n,d,\lambda)$ graphs \cite{han_spanning_2023,hyde_spanning_nodate}, tournaments \cite{benford_trees_2022,benford_trees_2022-1}, dense directed graphs \cite{kathapurkar_spanning_2022}, edge-colored graphs \cite{glock_decompositions_2021} and finite vector spaces\cite{chakraborti_almost_2024}.
This paper demonstrates that in edge-colored complete graphs, the absence of monochromatic triangle implies the existence of PC copies of every spanning tree, which is a subdivision of a $k$-star. The most challenging part in the proof is the case of multipartite tournaments. In fact, every $n$-vertex rooted tree oriented from each child-vertex to its father-vertex, is contained in a transitive tournament $D$ on $n$ vertices. Color each arc in $D$ with the label of its tail, resulting an edge-colored complete graph, which is mono-$C_3$-free and contains PC copies of every spanning tree. Motivated by this example, we ask the following question.
\begin{question}
Let $T$ be an $n$-vetrex tree with maximum degree at most $\Delta$ and let $G$ be a mono-$C_3$-free edge-colored $K_n$. Does $G$ always contain a PC copy of $T$ when $n$ is sufficiently large?
\end{question}
A graph $G$ is called $\mathcal{T}(n,\Delta)$-universal if $G$ contains every spanning tree with maximum degree at most $\Delta$. The statement of above question is also inspired by a question from Alon, Krivelevich and Sudakov\cite{alon_embedding_2007} on the $\mathcal{T}(n,\Delta)$-universal property of $(n,d,\lambda)$-graphs. We also observed that the result of Benford and Montgomery\cite{benford_trees_2022-1} may help dealing with the multipartite tournament case in further research on Conjecture \ref{conj:spanning tree}. It would be interesting to see more connections between Conjecture \ref{conj:spanning tree} and tree embedding results in graphs and digraphs.

~\\
~\\
{\bf Acknowledgements.} The authors would like to thank Xujiao Liu and Fangfang Wu for early discussions on this topic. Thanks also go to Fengming Dong and Donglei Yang for helpful comments. Part of this work was initiated soon after the first author began her visit to the Institute for Basic Science in Korea. She greatly appreciates the hospitality and the pleasant atmosphere in Extremal Combinatorics and Probability Group that contributed to the successful completion of this work.
\bibliographystyle{abbrv}
\bibliography{name}

\begin{thebibliography}{10}

\bibitem{aharoni_rainbow_2019}
R.~Aharoni, M.~DeVos, and R.~Holzman.
\newblock Rainbow triangles and the {Caccetta}-{H{\"a}ggkvist} conjecture.
\newblock {\em Journal of Graph Theory}, 92(4):347--360, 2019.

\bibitem{alon_embedding_2007}
N.~Alon, M.~Krivelevich, and B.~Sudakov.
\newblock Embedding nearly-spanning bounded degree trees.
\newblock {\em Combinatorica}, 27(6):629--644, 2007.

\bibitem{alon_random_2017}
N.~Alon, A.~Pokrovskiy, and B.~Sudakov.
\newblock Random subgraphs of properly edge-coloured complete graphs and long rainbow cycles.
\newblock {\em Israel Journal of Mathematics}, 222(1):317--331, 2017.

\bibitem{andersen_hamilton_1989}
L.~D. Andersen.
\newblock Hamilton circuits with many colours in properly edge-coloured
  complete graphs.
\newblock {\em Mathematica Scandinavica}, 64(1):5--14, 1989.

\bibitem{barr_properly_1998}
O.~Barr.
\newblock Properly coloured {Hamiltonian} paths in edge-coloured complete
  graphs without monochromatic triangles.
\newblock {\em Ars Combinatoria}, 50:316--318, 1998.

\bibitem{benford_trees_2022-1}
A.~Benford and R.~Montgomery.
\newblock Trees with few leaves in tournaments.
\newblock {\em Journal of Combinatorial Theory, Series B}, 155:141--170, 2022.

\bibitem{benford_trees_2022}
A.~Benford and R.~Montgomery.
\newblock Trees with many leaves in tournaments, 2022.
\newblock arXiv:2207.06384 [math].

\bibitem{bollobas_alternating_1976}
B.~Bollob\'{a}s and P.~Erd\H{o}s.
\newblock Alternating {Hamiltonian} cycles.
\newblock {\em Israel Journal of Mathematics}, 23(2):126--131, 1976.

\bibitem{bondy2008graph}
J.~A. Bondy and U.~S.~R. Murty.
\newblock {\em Graph theory}.
\newblock Springer Publishing Company, Incorporated, 2008.

\bibitem{chakraborti_almost_2024}
D.~Chakraborti and B.~Lund.
\newblock Almost spanning distance trees in subsets of finite vector spaces,
  2024.
\newblock arXiv:2306.12023 [math].

\bibitem{ding_properly_2022}
L.~Ding, J.~Hu, G.~Wang, and D.~Yang.
\newblock Properly colored short cycles in edge-colored graphs.
\newblock {\em European Journal of Combinatorics}, 100:103436, Feb. 2022.

\bibitem{erdos_combinatorial_1950}
P.~Erd\H{o}s and R.~Rado.
\newblock A combinatorial theorem.
\newblock {\em Journal of the London Mathematical Society}, 25:249--255, 1950.

\bibitem{fujita_decomposing_2019}
S.~Fujita, R.~Li, and G.~Wang.
\newblock Decomposing edge-coloured graphs under colour degree constraints.
\newblock {\em Combinatorics, Probability and Computing}, 28(5):755--767, 2019.

\bibitem{gallai_transitiv_1967}
T.~Gallai.
\newblock Transitiv orientierbare {Graphen}.
\newblock {\em Acta Mathematica. Academiae Scientiarum Hungaricae}, 18:25--66,
  1967.

\bibitem{glock_decompositions_2021}
S.~Glock, D.~K{\"u}hn, R.~Montgomery, and D.~Osthus.
\newblock Decompositions into isomorphic rainbow spanning trees.
\newblock {\em Journal of Combinatorial Theory, Series B}, 146:439--484, 2021.

\bibitem{gyarfas_ramsey_1987}
A.~Gy\'{a}rf\'{a}s, J.~Lehel, R.~H. Schelp, and Z.~Tuza.
\newblock Ramsey numbers for local colorings.
\newblock {\em Graphs and Combinatorics}, 3(1):267--277, 1987.

\bibitem{gyarfas_edge_2004}
A.~Gy\'{a}rf\'{a}s and G.~Simonyi.
\newblock Edge colorings of complete graphs without tricolored triangles.
\newblock {\em Journal of Graph Theory}, 46(3):211--216, 2004.

\bibitem{han_spanning_2023}
J.~Han and D.~Yang.
\newblock Spanning trees in sparse expanders, 2023.
\newblock arXiv:2211.04758 [math].

\bibitem{hyde_spanning_nodate}
J.~Hyde, N.~Morrison, A.~Muyesser, and M.~Pavez-Signe.
\newblock Spanning trees in pseudorandom graphs via sorting networks, 2023.
\newblock arXiv:2311.03185 [math].

\bibitem{jamison_constrained_2003}
R.~E. Jamison, T.~Jiang, and A.~C.~H. Ling.
\newblock Constrained {Ramsey} numbers of graphs.
\newblock {\em Journal of Graph Theory}, 42(1):1--16, 2003.

\bibitem{kathapurkar_spanning_2022}
A.~Kathapurkar and R.~Montgomery.
\newblock Spanning trees in dense directed graphs.
\newblock {\em Journal of Combinatorial Theory, Series B}, 156:223--249, 2022.

\bibitem{li_properly_2021}
R.~Li.
\newblock Properly colored cycles in edge-colored complete graphs without
  monochromatic triangle: {A} vertex-pancyclic analogous result.
\newblock {\em Discrete Mathematics}, 344(11):112573, 2021.

\bibitem{li_vertexdisjoint_2020}
R.~Li, H.~Broersma, and S.~Zhang.
\newblock Vertex-disjoint properly edge‐colored cycles in edge‐colored
  complete graphs.
\newblock {\em Journal of Graph Theory}, 94(3):476--493, 2020.

\bibitem{li_classification_2020}
R.~Li, B.~Li, and S.~Zhang.
\newblock A classification of edge-colored graphs based on properly colored
  walks.
\newblock {\em Discrete Applied Mathematics}, 283:590--595, 2020.

\bibitem{lo_properly_2016}
A.~Lo.
\newblock Properly coloured {Hamiltonian} cycles in edge-coloured complete
  graphs.
\newblock {\em Combinatorica}, 36(4):471--492, 2016.

\bibitem{montgomery_spanning_2019}
R.~Montgomery.
\newblock Spanning trees in random graphs.
\newblock {\em Advances in Mathematics}, 356:106793, 2019.

\bibitem{montgomery_spanning_nodate}
R.~Montgomery.
\newblock Spanning cycles in random directed graphs.
\newblock 2022.
\newblock arXiv:2103.06751 [math].

\end{thebibliography}
\end{document}